\documentclass{article}

\usepackage{fullpage}
\usepackage{graphicx} 
\usepackage{amsfonts}
\usepackage{amssymb} 
\usepackage{latexsym}
\usepackage{pifont}
\usepackage{amssymb,amsfonts}
\usepackage{amsthm}
\usepackage{amsmath}
\usepackage{graphics,graphicx,psfrag}
\usepackage{amscd}
\usepackage{epsfig}
\usepackage[all]{xy}
\usepackage{float}
\newtheorem{theorem}{Theorem}[section]
\newtheorem{proposition}{Proposition}[section]

\newtheorem{remark}{Remark}[section]
\newtheorem{lemma}{Lemma}[section]
\newtheorem{example}{Example}[section]

\newtheorem{corollary}{Corollary}[section]

\usepackage{framed}
\usepackage[usenames,dvipsnames]{color}
\usepackage{tikz}
\usepackage{xcolor}
\usetikzlibrary{matrix,arrows}
\usetikzlibrary{decorations.pathmorphing}
\usepackage[top=4.5cm, bottom=5.5cm, left=4.5cm, right=4cm]{geometry}

\usepackage{setspace}
\newcommand{\Mod}{ R{\rm {\bf -Mod}} }
\newcommand{\Ch}{ {\rm {\bf Ch}}(R{\rm {\bf -Mod}}) }

\begin{document}

\thispagestyle{empty}

\begin{center}
{\sc {\bf Degreewise $n$-projective and $n$-flat model structures on chain complexes } }
\end{center} 

\begin{center}
Marco A. P\'erez B. \\
Universit\'e du Qu\'ebec \`a Montr\'eal. \\
D\'epartement de Math\'ematiques. 
\end{center}

\begin{center} July 1st, 2012
\end{center}


\begin{abstract}
\noindent In \cite{Rada}, the authors construct an abelian model structure on the category  of chain complexes over left $R$-modules, $\Ch$, where the class of (trivially) cofibrant objects is given by the class of degreewise projective (resp. exact) chain complexes. Using a generalization of a well known theorem by I. Kaplansky, we generalize the method used in \cite{Rada} in order to obtain, for each integer $n > 0$, a new abelian model structure on the $\Ch$, where the class of (trivially) cofibrant objects is the class of (resp. exact) chain complexes whose terms have projective dimension at most $n$, provided the ring $R$ is right noetherian. We also give another method to construct this model structure, which also works to construct a model structure where the class of (trivially) cofibrant objects is given by the class of (resp. exact) chain complexes whose terms have flat dimension at most $n$.
\end{abstract}


\section{Introduction}

A cotorsion pair in an abelian category $\mathcal{C}$ is a pair $(\mathcal{A, B})$, where $\mathcal{A}$ and $\mathcal{B}$ are classes of objects of $\mathcal{C}$ such that they are orthogonal to each other with respect to the ${\rm Ext}$ functor. A model category is a bicomplete category with three classes of morphisms, called cofibrations, fibrations and weak equivalences, satisfying certain conditions. It turns out to be that these two notions have a deep connection. As far as the author knows, the first person who described this connection was M. Hovey in the paper {\it Cotorsion pairs, model category structures and representation theory}, where he proved that any two compatible and complete cotorsion pairs $(\mathcal{A}, \mathcal{B} \cap \mathcal{E})$ and $(\mathcal{A} \cap \mathcal{E}, \mathcal{B})$, in a bicomplete abelian category $\mathcal{C}$, give rise to a unique abelian model structure on $\mathcal{C}$ where $\mathcal{A}$ is the class of cofibrant objects, $\mathcal{B}$ is the class of fibrant objects, and $\mathcal{E}$ is the class of trivial objects.  \\

From this point there has been an increasing interest in constructing new model structures, specially on $\Ch$. One of the most influential researchers in this matter has been J. Gillespie, who has provided several results that allows us to induce cotorsion pairs in the category ${\rm Ch}(\mathcal{C})$ of chain complexes over an abelian category $\mathcal{C}$, from a certain cotorsion pair in $\mathcal{C}$. One of those results states that given a cotorsion pair $(\mathcal{A, B})$ in an abelian category $\mathcal{C}$ with enough projective and injective objects, there exist two cotorsion pairs in ${\rm Ch}(\mathcal{C})$ given by $({\rm dw}\widetilde{\mathcal{A}}, ({\rm dw}\widetilde{\mathcal{A}})^\perp)$ and $({\rm ex}\widetilde{\mathcal{A}}, ({\rm ex}\widetilde{\mathcal{A}})^\perp)$, where ${\rm dw}\widetilde{\mathcal{A}}$ is the class of chain complexes $X$ such that $X_m \in \mathcal{A}$ for every $m \in \mathbb{Z}$, and ${\rm ex}\widetilde{\mathcal{A}} = {\rm dw}\widetilde{\mathcal{A}} \cap \mathcal{E}$ where $\mathcal{E}$ is the class of exact complexes. As an example, if $\mathcal{P}_0$ denotes the class of projective modules in the category $\Mod$ of left $R$-modules, then the cotorsion pair $(\mathcal{P}_0, \Mod)$ induces two cotorsion pairs $({\rm dw}\widetilde{\mathcal{P}_0}, ({\rm dw}\widetilde{\mathcal{P}_0})^\perp)$ and $({\rm ex}\widetilde{\mathcal{P}_0}, ({\rm ex}\widetilde{\mathcal{P}_0})^\perp)$. In \cite{Rada}, the authors prove that these pairs are compatible and complete, with the help of a theorem by I. Kaplansky, namely that every projective module can be written as a direct sum of countably generated projective modules. Then, using \cite[Theorem 2.2]{Hovey2}, they get a new abelian model structure on $\Ch$ where the class of cofibrant objects is the class ${\rm dw}\widetilde{\mathcal{P}_0}$, which we shall call the class of {\bf degreewise projective complexes}. We shall refer to this model structure as the {\bf dw-projective model structure}. \\

In \cite{Aldrich} it is proven that if $\mathcal{P}_n$ denotes the class of left $R$-modules with projective dimension at most $n$, then $(\mathcal{P}_n, \mathcal{P}_n^\perp)$ is a complete and hereditary cotorsion pair. It follows we have two induced cotorsion pairs $({\rm dw}\widetilde{\mathcal{P}_n}, ({\rm dw}\widetilde{\mathcal{P}_n})^\perp)$ and $({\rm ex}\widetilde{\mathcal{P}_n}, ({\rm ex}\widetilde{\mathcal{P}_n})^\perp)$ in $\Ch$. Our goal is to prove that these two cotorsion pairs are complete for every $n > 0$, in order two obtain a new abelian model structure on $\Ch$ such that ${\rm dw}\widetilde{\mathcal{P}_n}$ is the class of cofibrant objects. \\

This paper is organized as follows. First, we recall some definitions and introduce the notation we shall use. Then, we shall give a \textquotedblleft generalization\textquotedblright \ of the Kaplansky Theorem in $\mbox{Mod-}R$ provided that $R$ is left noetherian. Specifically, we shall prove that every module of projective dimension $\leq n$ has a $\mathcal{P}_n^{\aleph_0}$-filtration, where $\mathcal{P}^{\aleph_0}_n$ is the set of all modules $M$ for which there exists an exact sequence \[ 0 \longrightarrow P_n \longrightarrow \cdots \longrightarrow P_1 \longrightarrow P_0 \longrightarrow M \longrightarrow 0 \] where $P_k$ is a countably generated projective module, for every $0 \leq k \leq n$. Using this result, we shall prove that $({\rm dw}\widetilde{\mathcal{P}_n}, ({\rm dw}\widetilde{\mathcal{P}_n})^\perp)$ and $({\rm ex}\widetilde{\mathcal{P}_n}, ({\rm ex}\widetilde{\mathcal{P}_n})^\perp)$ are complete cotorsion pairs. Then, we shall give another method to prove the previous result. The interesting thing of this other method, based on arguments appearing in the proof of \cite[Proposition 4.1]{Aldrich1}, is that it can be applied to show that $({\rm dw}\widetilde{\mathcal{F}_n}, ({\rm dw}\widetilde{\mathcal{F}_n})^\perp)$ and $({\rm ex}\widetilde{\mathcal{F}_n}, ({\rm ex}\widetilde{\mathcal{F}_n})^\perp)$ are complete cotorsion pairs, where $\mathcal{F}_n$ denotes the class of left $R$-modules having flat dimension at most $n$. At the end of the paper, we shall give some comments concerning the dw-$n$-projective and dw-$n$-flat model structures.


\section{Preliminaries}

This section is devoted to recall some notions and to introduce part of the notation we shall use throughout the paper. From now on, we work in the category $\Mod$ of left $R$-modules, and the category $\Ch$ of chain complexes over $\Mod$. Given a chain complex $X = (X_m)_{m \in \mathbb{Z}}$ with boundary maps $\partial^X_m : X_m \longrightarrow X_{m-1}$, we shall denote $Z_m(X) := {\rm Ker}(\partial^X_m)$. A chain complex $X$ is said to be {\bf exact} if $Z_m(X) = \partial_{m+1}(X_{m+1})$, for every $m \in \mathbb{Z}$. A chain complex $Y$ is said to be a {\bf subcomplex} of $X$ if there exists a monomorphism $i : Y \longrightarrow X$. Then we can define the {\bf quotient complex} $X / Y$ as the complex whose components are given by $(X / Y)_m = X_m / Y_m$ and whose boundary maps $\partial^{X / Y}_m : X_m / Y_m \longrightarrow X_{m-1} / Y_{m-1}$ are given by \[ x + Y_m \mapsto \partial^X_m(x) + Y_{m-1}. \]

Let $\mathcal{C}$ be either $\Mod$ or $\Ch$. Let $\mathcal{A}$ and $\mathcal{B}$ be two classes of objects in $\mathcal{C}$. The pair $(\mathcal{A, B})$ is called a {\bf cotorsion pair} in $\mathcal{C}$ if the following conditions are satisfied: 
\begin{itemize}
\item[ (1)] $\mathcal{A} = \mbox{}^{\perp}\mathcal{B} := \{ X \in {\rm Ob}(\mathcal{C}) \mbox{ / } {\rm Ext}^1(X, B) = 0 \mbox{ for every }B \in \mathcal{B} \}$.

\item[ (2)] $\mathcal{B} = \mathcal{A}^{\perp} := \{ X \in {\rm Ob}(\mathcal{C}) \mbox{ / } {\rm Ext}^1(A, X) = 0 \mbox{ for every }A \in \mathcal{A} \}$. 
\end{itemize} 

A cotorsion pair $(\mathcal{A, B})$ in $\mathcal{C}$ is said to be {\bf complete} if:
\begin{itemize}
\item[ (a)] $(\mathcal{A, B})$ has {\bf enough projectives}: for every object $X$ there exist objects $A \in \mathcal{A}$ and $B \in \mathcal{B}$, and a short exact sequence \[ 0 \longrightarrow B \longrightarrow A \longrightarrow X \longrightarrow 0. \]
\item[ (b)] $(\mathcal{A, B})$ has {\bf enough injectives}: for every object $X$ there exist objects $A' \in \mathcal{A}$ and $B' \in \mathcal{B}$, and a short exact sequence \[ 0 \longrightarrow X \longrightarrow B' \longrightarrow A' \longrightarrow 0. \]
\end{itemize}

A cotorsion pair $(\mathcal{A, B})$ is said to be {\bf cogenerated} by a set $\mathcal{S} \subseteq \mathcal{A}$ if $\mathcal{B} = \mathcal{S}^\perp$. There is a wide range of complete cotorsion pairs, thanks to the following result, known as the Eklof and Trlifaj Theorem. \\

\begin{theorem}{\rm \cite[Theorem 10]{Eklof}}\label{eklof} Every cotorsion pair in $\mathcal{C}$ cogenerated by a set is complete. \\
\end{theorem}

\begin{example} \
\begin{itemize}
\item[ (1)] If $\mathcal{P}_0$ denotes the class of projective modules, then $(\mathcal{P}_0, \Mod)$ is a cotorsion pair. Since every projective module is a direct summand of a free module, and $R$ is projective, one can show that $(\mathcal{P}_0, \Mod)$ is cogenerated by the set $\{ R \}$ and hence it is complete. 

\item[ (2)] Similarly, if $\mathcal{I}_0$ denotes the class of injective modules, then $(\Mod, \mathcal{I}_0)$ is a cotorsion pair. Using the Baer's Criterion, one can show that $(\Mod, \mathcal{I}_0)$ is cogenerated by the set of modules of the form $R/I$, where $I$ is a left ideal of $R$. So $(\Mod, \mathcal{I}_0)$ is a complete cotorsion pair. 

\item[ (3)] A less trivial example of a complete cotorsion pair is given by the {\bf flat cotorsion pair} $(\mathcal{F}_0, \mathcal{F}_0^\perp)$, where $\mathcal{F}_0$ is the class of flat modules. This result was proven by Edgard E. Enochs by using the Eklof and Trlifaj Theorem. Enochs proved that the pair $(\mathcal{F}_0, \mathcal{F}_0^\perp)$ is cogenerated by the set $\mathcal{S} = \{ S \in \mathcal{F}_0 \mbox{ : } {\rm Card}(S) \leq \kappa \}$, where $\kappa$ is an infinite cardinal with $\kappa \geq {\rm Card}(R)$. 

\item[ (4)] The following example is probably the most important cotorsion pair we shall consider in this paper, the pair $(\mathcal{P}_n, \mathcal{P}_n^\perp)$, where $\mathcal{P}_n$ is the class of modules which have projective dimension $\leq n$. Recall that a module $M$ has {\bf projective dimension} $\leq n$ is there exists an exact sequence \[ 0 \longrightarrow P_n \longrightarrow P_{n-1} \longrightarrow \cdots \longrightarrow P_1 \longrightarrow P_0 \longrightarrow M \longrightarrow 0, \] such that $P_k$ is a projective module, for every $0 \leq k \leq n$. Such a sequence is called a {\bf projective resolution of $M$ of length $n$}. We shall refer to the modules in $\mathcal{P}_n$ as {\bf $n$-projective modules}. In \cite{Aldrich}, the authors proved that $(\mathcal{P}_n, \mathcal{P}_n^\perp)$ is a cotorsion pair cogenerated by the set of all $n$-projective modules whose cardinality is less or equal than a given infinite cardinal $\kappa$ with $\kappa \geq {\rm Card}(R)$. \newpage

\item[ (5)] In a similar way, consider the class $\mathcal{F}_n$ of modules $M$ such that $M$ has flat dimension at most $n$, or equivalently, there is an exact sequence \[ 0 \longrightarrow F_n \longrightarrow \cdots \longrightarrow F_1 \longrightarrow F_0 \longrightarrow M \longrightarrow 0 \] where $F_k$ is a flat module, for every $0 \leq k \leq n$. This sequence is called a {\bf flat resolution of length $n$}. In \cite[Theorem 4.1.3]{Gobel}, it is proven that $(\mathcal{F}_n, \mathcal{F}^\perp_n)$ is a complete cotorsion pair. Here we shall give a easier proof of this fact. \\ 
\end{itemize}
\end{example}

Now we recall the notion of a model category. Given a category $\mathcal{C}$, a map $f$ in $\mathcal{C}$ is a {\bf retract} of a map $g$ in $\mathcal{C}$ if there is a commutative diagram of the form \\
\[ \begin{tikzpicture}
\matrix (m) [matrix of math nodes, row sep=2.5em, column sep=3em]
{ A & C & A \\ B & D & B \\ };
\path[-latex]
(m-1-1) edge (m-1-2) edge node[left] {$f$} (m-2-1)
(m-1-2) edge (m-1-3) edge node[left] {$g$} (m-2-2)
(m-1-3) edge node[right] {$f$} (m-2-3)
(m-2-1) edge (m-2-2)
(m-2-2) edge (m-2-3);
\end{tikzpicture} \]
where the horizontal composites are identities. Let $f : A \longrightarrow B$ and $g : C \longrightarrow D$ be two maps in $\mathcal{C}$. We shall say that $f$ has the {\bf left lifting property} with respect to $g$ (or that $g$ has the {\bf right lifting property} with respect to $f$) if for every pair of maps $u : A \longrightarrow C$ and $v : B \longrightarrow D$ with $g \circ u = v \circ f$, there exists a map $d : B \longrightarrow C$ such that $g \circ d = v$ and $d \circ f = u$. \\
\[ \begin{tikzpicture}
\matrix (m) [matrix of math nodes, row sep=1em, column sep=3em]
{ A & C & & & & A & C \\ & & \mbox{} & & \mbox{} & \\ B & D & & & & B & D \\ };
\path[-latex]
(m-1-1) edge node[left] {$f$} (m-3-1) edge node[above] {$u$} (m-1-2)
(m-1-6) edge node[left] {$f$} (m-3-6) edge node[above] {$u$} (m-1-7)
(m-3-1) edge node[below] {$v$} (m-3-2)
(m-1-2) edge node[right] {$g$} (m-3-2)
(m-3-6) edge node[below] {$v$} (m-3-7)
(m-1-7) edge node[right] {$g$} (m-3-7);
\path[dotted, ->]
(m-3-6) edge node[above, sloped] {$\exists \mbox{ } d$} (m-1-7);
\path [->, decoration={zigzag,segment length=4,amplitude=.9, post=lineto,post length=2pt},font=\scriptsize, line join=round] 
(m-2-3) edge[decorate] (m-2-5);
\end{tikzpicture} \]

A {\bf model category} is a bicomplete category $\mathcal{C}$ equipped with three classes of maps named {\bf cofibrations}, {\bf fibrations} and {\bf weak equivalences}, satisfying the following properties: 
\begin{itemize}
\item[ (1)] {\bf 3 for 2:} If $f$ and $g$ are maps of $\mathcal{C}$ such that $g\circ f$ is defined and two of $f$, $g$ and $g\circ f$ are weak equivalences, then so is the third.
\item[ (2)] If $f$ and $g$ are maps of $\mathcal{C}$ such that $f$ is a retract of $g$ and $g$ is a weak equivalence, cofibration, or fibration, then so is $f$.
\end{itemize}
Define a map to be a ${\bf trivial \ cofibration}$ if it is both a weak equivalence and a cofibration. Similarly, define a map to be a {\bf trivial fibration} if it is both a weak equivalence and a fibration.
\begin{itemize}
\item[ (3)] Trivial cofibrations have the left lifting property with respect to fibrations, and cofibrations have the left lifting property with respect to trivial fibrations.
\item[ (4)] Every map $f$ can be factored as $f = \alpha \circ \beta = \gamma \circ \delta$, where $\alpha$ (resp. $\delta$) is a cofibration (resp. fibration), and $\gamma$ (resp. $\beta$) is a trivial cofibration (resp. trivial fibration). \\\end{itemize} 

An object $X$ in $\mathcal{C}$ is called {\bf cofibrant} if the map $0 \longrightarrow X$ is a cofibration, {\bf fibrant} if the map $X \longrightarrow 1$ is a fibration, and {\bf trivial} if the map $0 \longrightarrow X$ is a weak equivalence, where $0$ and $1$ denote the initial and terminal objects of $\mathcal{C}$, respectively. \\

Given a bicomplete abelian category $\mathcal{C}$, a model structure on it is said to be {\bf abelian} if the following conditions are satisfied:
\begin{itemize}
\item[ (a)] A map is a cofibration if and only if it is a monomorphism with cofibrant cokernel.

\item[ (b)] A map if a fibration if and only if it is an epimorphism with fibrant kernel. 
\end{itemize}


\section{Degreewise $n$-projective complexes}

We begin this section with the notion of a filtration. Let $\mathcal{C}$ be either $\Mod$ or $\Ch$. Given an object $X \in \mathcal{C}$, by a {\bf filtration} of $X$ indexed by an ordinal $\lambda$ we shall mean a family $(X^\alpha \mbox{ : } \alpha < \lambda)$ of subobjects of $X$ such that:
\begin{itemize}
\item[ (1)] $X = \bigcup_{\alpha < \lambda} X_\alpha$.

\item[ (3)] $X^\alpha$ is a subobject of $X^{\alpha'}$ whenever $\alpha \leq \alpha'$.

\item[ (4)] $X^\beta = \bigcup_{\alpha < \beta} X^\alpha$ for any limit ordinal $\beta < \lambda$.
\end{itemize}
If $\mathcal{S}$ is some class of objects in $\Ch$, we say that a filtration $(X^\alpha \mbox{ : } \alpha < \lambda)$ of $X$ is a {\bf $\mathcal{S}$-filtration} if for each $\alpha + 1 < \lambda$ we have that $X_0$ and $X^{\alpha + 1} / X^\alpha$ are isomorphic to an element of $\mathcal{S}$. \\ 

The construction of the model structure given in \cite{Rada} is based on a theorem by I. Kaplansky, namely: 

\begin{theorem}[Kaplansky's Theorem] If $P$ is a projective module then $P$ is a direct sum of countable generated projective modules. \\
\end{theorem}

So when one thinks of a possible generalization of the dw-projective model structure for $n$-projective modules, a good question would be if it is possible to generalize the Kaplansky's Theorem for such modules. Let $M \in \mathcal{P}_n$ be an $n$-projective module: \[ 0 \longrightarrow P_n \longrightarrow \cdots \longrightarrow P_1 \longrightarrow P_0 \longrightarrow M \longrightarrow 0. \] By Kaplansky's Theorem we can write $P_k = \bigoplus_{i \in I_k} P^i_k$, where $P^i_k$ is a countably generated projective module, for every $i \in I_k$ and every $0 \leq k \leq n$. Then we can rewrite the previous resolution as \[ 0 \longrightarrow \bigoplus_{i \in I_n} P^i_n \longrightarrow \bigoplus_{i \in I_{n-1}} P^i_{n-1} \longrightarrow \cdots \longrightarrow \bigoplus_{i \in I_1} P^i_1 \longrightarrow \bigoplus_{i \in I_0} P^i_0 \longrightarrow M \longrightarrow 0. \] From now on we shall write any projective resolution of length $n$ by using such direct sum decompositions. We shall denote by $\mathcal{P}_n^{\aleph_0}$ the set of all modules $M$ having a projective resolution as above, where $I_k$ is a countable set for each $0 \leq k \leq n$. \\

For any class of modules $\mathcal{A}$, we denote by ${\rm dw}\widetilde{\mathcal{A}}$ (resp. ${\rm ex}\widetilde{\mathcal{A}}$) the class of (resp. exact) chain complexes such that each term belongs to $\mathcal{A}$. We shall prove that $({\rm dw}\widetilde{\mathcal{P}_n}, ({\rm dw}\widetilde{\mathcal{P}_n})^\perp)$ is a cotorsion pair cogenerated by the set ${\rm dw}\widetilde{\mathcal{P}_n^{\aleph_0}}$. We shall name ${\rm dw}\widetilde{\mathcal{P}_n}$ the class of {\bf dw-$n$-projective chain complexes}. The fact that $({\rm dw}\widetilde{\mathcal{P}_n}, ({\rm dw}\widetilde{\mathcal{P}_n})^\perp)$ is a cotorsion pair in $\Ch$ is a consequence of the following result (which is proven by its author for any abelian category): 

\begin{proposition}{\rm \cite[Proposition 3.2]{Gillespie}} Let $(\mathcal{A, B})$ be a cotorsion pair in $\Mod$. Then $({\rm dw}\widetilde{\mathcal{A}}, ({\rm dw}\widetilde{\mathcal{A}})^\perp)$ is a cotorsion pair in $\Ch$. \\
\end{proposition}

We shall prove that every dw-$n$-projective complex has a ${\rm dw}\widetilde{\mathcal{P}_n^{\aleph_0}}$-filtration. Then the completeness of $({\rm dw}\widetilde{\mathcal{P}_n}, ({\rm dw}\widetilde{\mathcal{P}_n})^\perp)$ shall be a consequence of Theorem \ref{eklof} and the following result: 

\begin{proposition}\label{propo} Let $(\mathcal{A, B})$ be a cotorsion pair in $\mathcal{C} = \Mod, \Ch$ and let $\mathcal{S} \subseteq \mathcal{A}$ be a set of objects of $\mathcal{C}$. If every $A \in \mathcal{A}$ has a $\mathcal{S}$-filtration, then $(\mathcal{A, B})$ is cogenerated by $\mathcal{S}$. \\
\end{proposition}

Before proving this, we need the following result known as the Eklof's Lemma. For a proof of this we refer the reader to \cite[Lemma 3.1.2]{Gobel} or \cite[Theorem 7.3.4]{Enochs}. 

\begin{lemma}[Eklof's Lemma] In $\mathcal{C} = \Mod, \Ch$ let $A$ and $B$ be two objects. If $A$ has a $\mbox{}^\perp\{ B \}$-filtration, then $A \in \mbox{}^\perp\{B\}$. \\
\end{lemma}

\begin{proof}[{\rm {\bf Proof of Proposition \ref{propo}:}}] Consider the cotorsion pair $(\mbox{}^\perp(\mathcal{S}^\perp), \mathcal{S}^\perp)$. We shall show that $(\mbox{}^\perp({S}^\perp), \mathcal{S}^\perp) = (\mathcal{A, B})$. It suffices to show that $\mathcal{B} = \mathcal{S}^\perp$, since this equality implies $\mathcal{A} = \mbox{}^\perp\mathcal{B} = \mbox{}^\perp(\mathcal{S}^\perp)$. Since $\mathcal{S} \subseteq \mathcal{A}$, we have $\mathcal{B} = \mathcal{A}^\perp \subseteq \mathcal{S}^\perp$. Now let $Y \in \mathcal{S}^\perp$, $A \in \mathcal{A}$ and let $(A_\alpha \mbox{ : } \alpha < \lambda)$ be an $\mathcal{S}$-filtration of $A$. We have 
\begin{align*}
{\rm Ext}^1(A_0, Y) & = {\rm Ext}^1(0, Y) = 0, \\
{\rm Ext}^1(A_{\alpha + 1} / A_\alpha, Y) & = , \mbox{ whenever $\alpha + 1 < \lambda$},
\end{align*}
since $A_0$ and $A_{\alpha + 1} / A_\alpha$ are isomorphic to objects in $\mathcal{C}$. Then $(A_\alpha \mbox{ : } \alpha < \lambda)$ is a $\mbox{}^\perp\{Y\}$-filtration of $A$. By the Eklof's Lemma, we have ${\rm Ext}^1(A, Y) = 0$, i.e. $Y \in \mathcal{A}^{\perp} = \mathcal{B}$ since $A$ is any module in $\mathcal{A}$. Hence $\mathcal{S}^\perp \subseteq \mathcal{B}$. \\
\end{proof}

In order to construct ${\rm dw}\widetilde{\mathcal{P}_n^{\aleph_0}}$-filtrations of dw-$n$-projective complexes, we need the following generalization of the Kaplansky's Theorem: 

\begin{lemma}[Kaplansky's Theorem fon $n$-projective modules]\label{Kaplansky} Let $R$ be a noetherian ring. Let $M \in \mathcal{P}_n$ and let $N$ be a countably generated submodule of $M$. Then there exists a $\mathcal{P}^{\aleph_0}_n$-filtration of $M$, say $(M_\alpha : \alpha < \lambda)$  with $\lambda > 1$, such that $M_1 \in \mathcal{P}^{\aleph_0}_n$ and $N \subseteq M_1$. 
\end{lemma}
\begin{proof} Let $M \in \mathcal{P}_n$ and let \[ 0 \longrightarrow \bigoplus_{i \in I_n} P^i_n \longrightarrow \bigoplus_{i \in I_{n-1}} P^i_{n-1} \longrightarrow \cdots \longrightarrow \bigoplus_{i \in I_1} P^i_1 \longrightarrow \bigoplus_{i \in I_0} P^i_0 \longrightarrow M \longrightarrow 0 \] be a projective resolution of $M$. We shall construct a $\mathcal{P}^{\aleph_0}_n$-filtration $(M_\alpha \mbox{ : } \alpha < \lambda)$ of $M$, with $N \subseteq M_1$, by using transfinite induction. For $\alpha = 0$ set $M_0 = 0$. Now we construct $M_1$. Let $\mathcal{G}$ be a countable set of generators of $N$. Since $f_0$ is surjective, for every $g \in \mathcal{G}$ we can choose $y_g \in \bigoplus_{i \in I_0} P^i_0$ such that $g = f_0(y_g)$. Consider the set $Y = \{ y_g \mbox{ : } g \in \mathcal{G} \}$. Since $Y$ is a countable subset of $\bigoplus_{i \in I_0} P^i_0$, we have that $\left< Y \right>$ is a countably generated submodule of $P_0$. Choose a countable subset $I_0^{1, 0} \subseteq I_0$ such that $\left< Y \right> \subseteq \bigoplus_{i \in I^{1, 0}_0} P^i_0$. Then $f_0\left( \left< Y \right> \right) \subseteq N$. Consider ${\rm Ker}\left( \left. f_0 \right|_{\bigoplus_{i \in I_0^{1, 0}} P^i_0} \right)$. Since $\bigoplus_{i \in I_0^{1, 0}} P^i_0$ is countably generated and ${\rm Ker}\left( \left. f_0 \right|_{\bigoplus_{i \in I_0^{1, 0}} P^i_0} \right)$ is a submodule of $\bigoplus_{i \in I_0^{1, 0}} P^i_0$, we have that ${\rm Ker}\left( \left. f_0 \right|_{\bigoplus_{i \in I_0^{1, 0}} P^i_0} \right)$ is also countably generated, since $R$ is noetherian. Let $\mathcal{B}$ be a countable set of generators of ${\rm Ker}\left( \left. f_0 \right|_{\bigoplus_{i \in I_0^{1, 0}} P^i_0} \right)$. Let $b \in \mathcal{B}$, then $f(b) = 0$ and by exactness of the above sequence there exists $y_b \in \bigoplus_{i \in I_1} P^i_1$ such that $b = f_1(y_b)$. Let $Y' = \{ y_b \mbox{ : } b \in \mathcal{B} \}$. Note that $Y'$ is a countable subset of $(f_1)^{-1}\left( {\rm Ker}\left( \left. f_0 \right|_{\bigoplus_{i \in I_0^{1, 0}} P^i_0} \right) \right)$. Then $\left< Y' \right>$ is a countably generated submodule of $\bigoplus_{i \in I_1} P^i_1$. Hence there exists a countable subset $I_1^{1, 0} \subseteq I_1$ such that $\bigoplus_{i \in I_1^{1, 0}} P^i_1 \supseteq \left< Y' \right>$. Thus $f_1 \left( \bigoplus_{i \in I^{1, 0}_1} P^i_1 \right) \supseteq f_1(\left< Y' \right>)$. Now let $z \in {\rm Ker}\left( \left. f_0 \right|_{\bigoplus_{i \in I_0^{1, 0}} P^i_0} \right)$. Then $z = r_1 b_1 + \cdots + r_m b_m$, where each $b_j \in \mathcal{B}$. Since $b_j = f_1(y_{b_j})$ with $y_{b_j} \in Y'$, we get $z = f_1(r_1y_{b_1} + \cdots + r_m y_{b_m}) \in f_1(\left< Y' \right>)$. Hence, ${\rm Ker}\left( \left. f_0 \right|_{\bigoplus_{i \in I_0^{1, 0}} P^i_0} \right) \subseteq f_1(\left< Y' \right>) \subseteq f_1 \left( \bigoplus_{i \in I_1^{1, 0}} P^i_1 \right)$. Use the same argument to find a countable subset $I_2^{1, 0} \subseteq I_2$ such that $f_2\left( \bigoplus_{i \in I_2^{1, 0}} P^i_2 \right) \supseteq {\rm Ker}\left( f_1|_{\bigoplus_{i \in I_1^{1, 0}} P^i_1} \right)$. Repeat the same argument until find a countable subset $I_n^{1, 0} \subseteq I_n$ such that $f_n \left( \bigoplus_{i \in I_n^{1, 0}} P^i_n \right) \supseteq {\rm Ker}\left( f_{n-1}|_{\bigoplus_{i \in I_{n-1}^{1, 0}} P^i_{n-1}} \right)$. Now, $f_n\left( \bigoplus_{i \in I_n^{1,0}} P^i_n \right)$ is a countably generated submodule of $\bigoplus_{i \in I_{n-1}} P^i_{n-1}$. Then choose a countable subset $I_{n-1}^{1,0} \subseteq I_{n-1}^{1, 1} \subseteq I_{n-1}$ such that $f_n \left( \bigoplus_{i \in I_n^{1, 0}} P^i_n \right) \subseteq \bigoplus_{i \in I_{n-1}^{1,1}} P^i_{n-1}$. Repeat this process until find a countable subset $I^{1,0}_0 \subseteq I^{1,1}_0 \subseteq I_0$ satisfying $f_1\left( \bigoplus_{i \in I_1^{1,1}} P^i_1 \right)$ $\subseteq \bigoplus_{i \in I_0^{1,1}} P^i_0$. Now choose a countable subset $I_1^{1,1} \subseteq I_1^{1,2} \subseteq I_1$ such that $f_1\left( \bigoplus_{i \in I_2^{1, 2}} P^i_1 \right)$ $\supseteq {\rm Ker}\left( f_0|_{\bigoplus_{i \in I_0^{1,1}} P^i_0} \right)$. What we have been doing so far is called the zig-zag procedure. Keep repeating this procedure infinitely many times, and set $I^1_k = \bigcup_{m \geq 0} I_k^{1, m}$, for every $0 \leq k \leq n$. By construction, we get the following exact sequence \[ 0 \longrightarrow \bigoplus_{i\in I_n^1} P^i_n \longrightarrow \bigoplus_{i \in I_{n-1}^1} P^i_{n-1} \longrightarrow \cdots \longrightarrow \bigoplus_{i \in I^1_1} P^i_1 \longrightarrow \bigoplus_{i \in I^1_0} P^i_0 \longrightarrow M_1 \longrightarrow 0 \] where $x \in M_1 := {\rm CoKer}\left( \bigoplus_{i \in I^1_1} \longrightarrow \bigoplus_{i \in I_0^1} P^i_0 \right) \subseteq M$ and $N \subseteq M_1$. We take the quotient of the resolution of $M$ by the resolution of $M'$, and get the following commutative diagram: \\
\[ \begin{tikzpicture}
\matrix (m) [matrix of math nodes, row sep=2em, column sep=1.5em]
{ & 0 & & 0 & 0 & 0 \\ 0 & \bigoplus_{i \in I^1_n} P^i_n & \cdots & \bigoplus_{i \in I^1_1} P^i_1 & \bigoplus_{i \in I^1_0} P^i_0 & M_1 & 0 \\ 0 & \bigoplus_{i \in I_n} P^i_n & \cdots & \bigoplus_{i \in I_1} P^i_1 & \bigoplus_{i \in I_0} P^i_0 & M & 0 \\ 0 & \bigoplus_{i \in I_n - I^1_n} P^i_n & \cdots & \bigoplus_{i \in I_1 - I^1_1} P^i_1 & \bigoplus_{i \in I_0 - I^1_0} P^i_0 & M / M_1 & 0 \\ & 0 & & 0 & 0 & 0 \\ };
\path[->]
(m-1-2) edge (m-2-2) (m-1-4) edge (m-2-4) (m-1-5) edge (m-2-5) (m-1-6) edge (m-2-6) 
(m-2-1) edge (m-2-2) (m-2-2) edge (m-2-3) edge (m-3-2) (m-2-3) edge (m-2-4) (m-2-4) edge (m-2-5) edge (m-3-4) (m-2-5) edge (m-2-6) edge (m-3-5) (m-2-6) edge (m-2-7) edge (m-3-6)
(m-3-1) edge (m-3-2) (m-3-2) edge (m-3-3) edge (m-4-2) (m-3-3) edge (m-3-4) (m-3-4) edge (m-3-5) edge (m-4-4) (m-3-5) edge (m-3-6) edge (m-4-5) (m-3-6) edge (m-3-7) edge (m-4-6)
(m-4-1) edge (m-4-2) (m-4-2) edge (m-4-3) edge (m-5-2) (m-4-3) edge (m-4-4) (m-4-4) edge (m-4-5) edge (m-5-4) (m-4-5) edge (m-4-6) edge (m-5-5) (m-4-6) edge (m-4-7) edge (m-5-6);
\end{tikzpicture} \] 
where the third row is an exact sequence since the class of exact complexes is closed under taking cokernels. Then we have a projective resolution of length $n$ for $M / M_1$. Repeat the same procedure above for $M / M_1$, by choosing $x^1 + M_1 \in M/M_1 - \left\{ 0 + M_1 \right\}$, in order to get an exact sequence \\ \[ 0 \longrightarrow \bigoplus_{i \in I^2_n - I^1_n} P^i_n \longrightarrow \cdots \longrightarrow \bigoplus_{i \in I^2_1 - I^1_1} P^i_1 \longrightarrow \bigoplus_{i \in I^2_0 - I^1_0} P^i_0 \longrightarrow M_2 / M_1 \longrightarrow 0, \] for some module $M_1 \subseteq M_2 \subseteq M$, such that $I^2_k - I^1_k$ is countable for every $0 \leq k \leq n$. Note that \[ 0 \longrightarrow \bigoplus_{i \in I^2_n} P^i_n \longrightarrow \bigoplus_{i \in I^2_{n-1}} P^i_{n-1} \longrightarrow \cdots \longrightarrow \bigoplus_{i \in I^2_1} P^i_1 \longrightarrow \bigoplus_{i \in I^2_0} P^i_0 \longrightarrow M_2 \longrightarrow 0 \] is a projective resolution of $M_2$, since we have a commutative diagram \\
\[ \begin{tikzpicture}
\matrix (m) [matrix of math nodes, row sep=2em, column sep=1.5em]
{ & 0 & & 0 & 0 & 0 \\ 0 & \bigoplus_{i \in I^1_n} P^i_n & \cdots & \bigoplus_{i \in I^1_1} P^i_1 & \bigoplus_{i \in I^1_0} P^i_0 & M_1 & 0 \\ 0 & \bigoplus_{i \in I^2_n} P^i_n & \cdots & \bigoplus_{i \in I^2_1} P^i_1 & \bigoplus_{i \in I^2_0} P^i_0 & M_2 & 0 \\ 0 & \bigoplus_{i \in I^2_n - I^1_n} P^i_n & \cdots & \bigoplus_{i \in I^2_1 - I^1_1} P^i_1 & \bigoplus_{i \in I^2_0 - I^1_0} P^i_0 & M_2 / M_1 & 0 \\ & 0 & & 0 & 0 & 0 \\ };
\path[->]
(m-1-2) edge (m-2-2) (m-1-4) edge (m-2-4) (m-1-5) edge (m-2-5) (m-1-6) edge (m-2-6) 
(m-2-1) edge (m-2-2) (m-2-2) edge (m-2-3) edge (m-3-2) (m-2-3) edge (m-2-4) (m-2-4) edge (m-2-5) edge (m-3-4) (m-2-5) edge (m-2-6) edge (m-3-5) (m-2-6) edge (m-2-7) edge (m-3-6)
(m-3-1) edge (m-3-2) (m-3-2) edge (m-3-3) edge (m-4-2) (m-3-3) edge (m-3-4) (m-3-4) edge (m-3-5) edge (m-4-4) (m-3-5) edge (m-3-6) edge (m-4-5) (m-3-6) edge (m-3-7) edge (m-4-6)
(m-4-1) edge (m-4-2) (m-4-2) edge (m-4-3) edge (m-5-2) (m-4-3) edge (m-4-4) (m-4-4) edge (m-4-5) edge (m-5-4) (m-4-5) edge (m-4-6) edge (m-5-5) (m-4-6) edge (m-4-7) edge (m-5-6);
\end{tikzpicture} \] 
where the first and third rows are exact sequences, and then so is the second since the class of exact complexes is closed under extensions. We have that $M_1$ and $M_2$ are $n$-projective modules such that $M_1 \in \mathcal{P}^{\aleph_{0}}_n, M_2 / M_1 \in \mathcal{P}^{\aleph_0}_n$. Now suppose that there is an ordinal $\beta$ such that:
\begin{itemize}
\item[ (1)] $M_\alpha$ is an $n$-projective module, for every $\alpha < \beta$.

\item[ (2)] $M_\alpha \subseteq M_{\alpha'}$ whenever $\alpha \leq \alpha' < \beta$. 

\item[ (3)] $M_{\alpha + 1} / M_\alpha \in \mathcal{P}^{\aleph_0}_n$ whenever $\alpha + 1 < \beta$. 

\item[ (4)] $M_\gamma = \bigcup_{\alpha < \gamma} M_\alpha$ for every limit ordinal $\gamma < \beta$. 
\end{itemize} 
If $\beta$ is a limit ordinal, then set $M_\beta = \bigcup_{\alpha < \beta} M_\alpha$. Otherwise there exists an ordinal $\alpha < \beta$ such that $\beta = \alpha + 1$. In this case, construct $M_{\alpha + 1} \in \mathcal{P}_n$ from $M_\alpha$ as we constructed $M_2$ from $M_1$, such that $M_{\alpha + 1} / M_\alpha \in \mathcal{P}^{\aleph_0}_n$. By transfinite induction, we obtain a $\mathcal{P}_n^{\aleph_0}$-filtration $(M_\alpha \mbox{ : } \alpha < \lambda)$ of $M$, for some ordinal $\lambda$, such that $M_1 \supseteq N$ and $M_1 \in \mathcal{P}^{\aleph_0}_n$. \\
\end{proof}

From now on, $R$ shall be a noetherian ring. Now we are ready to prove the main result of this section. \\

\begin{theorem} Every chain complex $X \in {\rm dw}\widetilde{\mathcal{P}_n}$ has a ${\rm dw}\widetilde{\mathcal{P}_n^{\aleph_0}}$-filtration. 
\end{theorem}
\begin{proof} Let $X \in {\rm dw}\mathcal{P}_n$ and write \[ X = \cdots \longrightarrow X_{k+1} \stackrel{\partial_{k+1}}\longrightarrow X_k \stackrel{\partial_k}\longrightarrow X_{k-1} \longrightarrow \cdots. \] For each $k$ one has a projective resolution of $X_k$ of length $n$: \[ 0 \longrightarrow \bigoplus_{i \in I_n(k)} P^i_n(k) \longrightarrow \cdots \longrightarrow \bigoplus_{i \in I_1(k)} P^i_1(k) \longrightarrow \bigoplus_{i \in I_0(k)} P^i_0(k) \longrightarrow X_k \longrightarrow 0. \] \\ We shall construct a ${\rm dw}\widetilde{\mathcal{P}_n^{\aleph_0}}$-filtration of $X$ by using transfinite induction. For $\alpha = 0$ set $X^0 = 0$. For $\alpha = 1$, choose $m \in \mathbb{Z}$. Let $S$ be a countably generated submodule of $X_m$. By the previous lemma, there exists a submodule $\mathcal{P}^{\aleph_0}_n \ni X^1_m \subseteq X_m$ such that $S \subseteq X^1_m$. Note that $X^1_m$ is also countably generated. Then $\partial_m(X^1_m)$ is a countably generated submodule of $X_{m-1}$, and so there exists $\mathcal{P}^{\aleph_0}_n \ni X^1_{m-1} \subseteq X_{m-1}$ such that $\partial_m(X^1_m) \subseteq X^1_{m-1}$. Repeat the same procedure infinitely many times in order to obtain a subcomplex \[ X^1 = \cdots \longrightarrow X^1_{k+1} \longrightarrow X^1_k \longrightarrow X^1_{k-1} \longrightarrow \cdots \] of $X$ such that $X^1_k \in \mathcal{P}^{\aleph_0}_n$ for every $k \in \mathbb{Z}$ (we are setting $X^1_k = 0$ for every $k > m$). Hence $X^1 \in {\rm dw}\widetilde{\mathcal{P}^{\aleph_0}_n}$. Note from the proof of the previous lemma that the quotient $X / X^1$ is in ${\rm dw}\widetilde{\mathcal{P}_n}$. We have \[ X / X^1 = \cdots \longrightarrow X_{k+1} / X^1_{k+1} \longrightarrow X_k / X^1_k \longrightarrow X_{k-1} / X^1_{k-1} \longrightarrow \cdots, \] where for every $k \leq m$ one has the following projective resolutions of length $n$ for $X^1_k$ and $X_k / X^1_{k}$:  
\begin{align*}
0 & \longrightarrow \bigoplus_{i \in I_n^1(k)} P^i_n(k) \longrightarrow \cdots \longrightarrow \bigoplus_{i \in I^1_1(k)} P^i_1(k) \longrightarrow \bigoplus_{i \in I^1_0(k)} P^i_0(k) \longrightarrow X^1_k \longrightarrow 0, \\
0 & \longrightarrow \bigoplus_{i \in I_n(k) - I_n^1(k)} P^i_n(k) \longrightarrow \cdots \longrightarrow \bigoplus_{i \in I_0(k) - I^1_0(k)} P^i_0(k) \longrightarrow X_k / X^1_k \longrightarrow 0.
\end{align*}
Apply the same procedure above to the complex $X / X^1$, in order to get a subcomplex \[ X^2 / X^1 = \cdots \longrightarrow X^2_{k+1} / X^1_{k+1} \longrightarrow X^2_k / X^1_k \longrightarrow X^2_{k-1} / X^1_{k-1} \longrightarrow \cdots \] of $X / X^1$, such that for each $k \in \mathbb{Z}$ one has the following projective resolution of length $n$ for $X^2_k / X^1_k$: \\ \[ 0 \longrightarrow \bigoplus_{i \in I^2_n - I^1_n} P^i_n(k) \longrightarrow \cdots \longrightarrow \bigoplus_{i \in I^2_1 - I^1_1} P^i_1(k) \longrightarrow \bigoplus_{i \in I^2_0 - I^1_0} P^i_0(k) \longrightarrow X^2_k / X^1_k \longrightarrow 0, \] where each $I^2_j - I^1_j \subseteq I_j$ is countable. Now consider the complex \[ X^2 = \cdots \longrightarrow X^2_{k+1} \longrightarrow X^2_k \longrightarrow X^2_{k-1} \longrightarrow \cdots. \] As we did in the proof of the previous lemma, we have that \[ 0 \longrightarrow \bigoplus_{i \in I^2_n(k)} P^i_n(k) \longrightarrow \cdots \longrightarrow \bigoplus_{i \in I^2_1(k)} P^i_1(k) \longrightarrow \bigoplus_{i \in I^2_0(k)} P^i_0(k) \longrightarrow X^2_k \longrightarrow 0 \] is an exact sequence. So $X^2_k \in \mathcal{P}_n$ for every $k \in \mathbb{Z}$, and hence $X^2 \in {\rm dw}\widetilde{\mathcal{P}_n}$, with $X^2 / X^1 \in {\rm dw}\widetilde{\mathcal{P}^{\aleph_0}_n}$. The rest of the proof follows by transfinite induction, as in the end of the proof of the previous lemma.
\end{proof}


\section{Exact degreewise $n$-projective complexes} 

Consider the class of exact dw-$n$-projective complexes ${\rm ex}\widetilde{\mathcal{P}_n} = {\rm dw}\widetilde{\mathcal{P}_n} \cap \mathcal{E}$, where $\mathcal{E}$ denotes the class of exact complexes. The goal of this section is to prove that $({\rm ex}\widetilde{\mathcal{P}_n}, ({\rm ex}\widetilde{\mathcal{P}_n})^\perp)$ is a complete cotorsion pair. This pair is a cotorsion pair by the following result by Gillespie:
\begin{proposition}{\rm \cite[Proposition 3.3]{Gillespie}} Let $(\mathcal{A, B})$ be a cotorsion pair in an abelian category $\mathcal{C}$ with enough projective and injective objects. If $\mathcal{B}$ contains a cogenerator of finite injective dimension then $({\rm ex}\widetilde{\mathcal{A}}, ({\rm ex}\widetilde{\mathcal{A}})^\perp)$ is a cotorsion pair. \\
\end{proposition}

Recall that a {\bf cogenerator} in an abelian category $\mathcal{C}$ is an object $C$ such that for every nonzero object $H$ there exists a nonzero morphism $f : H \longrightarrow C$. For example, $\mbox{Mod-}R$ has a an injective cogenerator given by the abelian group ${\rm Hom}(R, \mathbb{Q} / \mathbb{Z})$ of group homomorphisms and providing it with the scalar multiplication defined by \[ f \cdot r : R \longrightarrow \mathbb{Q} / \mathbb{Z}, \mbox{ } s \mapsto f(rs) \] for $f \in {\rm Hom}(R, \mathbb{Q} / \mathbb{Z})$ and $r \in R$ (see \cite[Proposition 4.7.5]{Borceaux} for details). Since ${\rm Hom}(R, \mathbb{Q} / \mathbb{Z}) \in \mathcal{P}_n^\perp$, we have that $\left( {\rm ex}\mathcal{P}_n, ({\rm ex}\mathcal{P}_n)^\perp \right)$ is a cotorsion pair. \\

Given a module $M \in \mathcal{P}_n$, consider a projective resolution of $M$ of length $n$: \[ 0 \longrightarrow \bigoplus_{i \in I_n} P^i_n \longrightarrow \bigoplus_{i \in I_{n-1}} P^i_{n-1} \longrightarrow \cdots \longrightarrow \bigoplus_{i \in I_1} P^i_1 \longrightarrow \bigoplus_{i \in I_0} P^i_0 \longrightarrow M \longrightarrow 0 \mbox{ \ ($\ast$)}. \] We shall say that a projective resolution \[ 0 \longrightarrow \bigoplus_{i \in I'_n} P^i_n \longrightarrow \bigoplus_{i \in I'_{n-1}} P^i_{n-1} \longrightarrow \cdots \longrightarrow \bigoplus_{i \in I'_1} P^i_1 \longrightarrow \bigoplus_{i \in I'_0} P^i_0 \longrightarrow N \longrightarrow 0 \mbox{ \ ($\ast \ast$)} \] is a {\bf nice subresolution} of ($\ast$) if $I'_k \subseteq I_k$ for every $0\leq k \leq n$ and $N \subseteq M$. \\

From now on, fix an infinite cardinal $\kappa$ such that $\kappa \geq {\rm Card}(R)$. We shall say that a set $S$ is {\bf small} if ${\rm Card}(S) \leq \kappa$. We shall also say that a chain complex $X = (X_m)_{m \in \mathbb{Z}}$ is {\bf small} if ${\rm Card}(X) \leq \kappa$, where \[ {\rm Card}(X) := \sum_{m \in \mathbb{Z}} {\rm Card}(X_m). \] So a complex $X$ is small if and only if each term $X_m$ is a small set. Note that if $M$ is an $n$-projective module with a resolution given by ($\ast$), then it is small if and only if ${\rm Card}(I_k) \leq \kappa$ for every $0 \leq k \leq n$. Let $\mathcal{P}^{\leq \kappa}_n$ denote the set of $n$-projective modules with a small projective resolution. If we consider the resolutions ($\ast$) and ($\ast \ast$) above, then note that ($\ast \ast$) is a small and nice subresolution of ($\ast$) if each $I'_k$ is a small subset of $I_k$. Consider the set \[ {\rm ex}\widetilde{\mathcal{P}_n^{\leq \kappa}} = \{ X \in \Ch \mbox{ : } X \mbox{ is exact and $X_m \in \mathcal{P}^{\leq \kappa}_n$ for every $m \in \mathbb{Z}$} \}. \] We shall prove that every exact dw-$n$-projective complex has a ${\rm ex}\widetilde{\mathcal{P}^{\leq \kappa}_n}$-filtration. \\

\begin{lemma}\label{lemma2} Let $M \in \mathcal{P}_n$ with a projective resolution given by ($\ast$). For every submodule $N \subseteq M$ with ${\rm Card}(N) \leq \kappa$, there exists a small and nice subresolution \[ 0 \longrightarrow \bigoplus_{i \in I'_n} P^i_n \longrightarrow \bigoplus_{i \in I'_{n-1}} P^i_{n-1} \longrightarrow \cdots \longrightarrow \bigoplus_{i \in I'_1} P^i_1 \longrightarrow \bigoplus_{i \in I'_0} P^i_0 \longrightarrow N' \longrightarrow 0 \mbox{ \ ($\ast \ast \ast$)} \] of ($\ast$) such that $N \subseteq N'$. Moreover, if $N$ has an small and nice subresolution of $M$, then ($\ast \ast \ast$) can be constructed in such a way that it contains the given resolution of $N$. 
\end{lemma}
\begin{proof} Since $f_0$ is surjective, for every $x \in N$ choose $y_x \in \bigoplus_{i \in I_0} P^i_0$ such that $x = f_0(y_x)$. Let $Y = \{ y_x \mbox{ : } x \in N \}$. Note that $\left< Y \right>$ is a small submodule of $\bigoplus_{i \in I_0} P^i_0$. So there exists  a small subset $I^0_0 \subseteq I_0$ such that $\left< Y \right> \subseteq \bigoplus_{i \in I^0_0} P^i_0$. We have $f_0\left( \bigoplus_{i \in I^0_0} P^i_0 \right) \supseteq N$. Now consider the submodule ${\rm Ker}\left( f_0|_{\bigoplus_{i \in I^0_0}} P^i_0 \right)$ of $f_0\left( \bigoplus_{i \in I^0_0} P^i_0 \right)$, which is small since $f_0\left( \bigoplus_{i \in I^0_0} P^i_0 \right)$ is. Then we can choose a small subset $I^0_1 \subseteq I_1$ such that $f_1\left( \bigoplus_{i \in I^0_1} P^i_1 \right) \supseteq {\rm Ker}\left( f_0|_{\bigoplus_{i \in I^0_0} P^i_0} \right)$. Repeat the same argument until find a small subset $I^0_n \subseteq I_n$ such that $f_n\left( \bigoplus_{i \in I^0_n} P^i_n \right) \supseteq {\rm Ker}\left( f_{n-1}|_{\bigoplus_{i \in I^0_{n-1}} P^i_{n-1}} \right)$. Since $f_n\left( \bigoplus_{i \in I^0_n} P^i_n \right)$ is a small submodule of $\bigoplus_{i \in I_{n-1}} P^i_{n-1}$, we can choose a small subset $I^0_{n-1} \subseteq I^1_{n-1} \subseteq I_{n-1}$ such that $f_n\left( \bigoplus_{i \in I^0_n} P^i_n \right) \subseteq \bigoplus_{i \in I^1_{n-1}} P^i_{n-1}$. From this point just use the zig-zag procedure in order to get small subsets $I'_k = \bigcup_{j \geq 0} I^j_k \subseteq I_k$ and an exact sequence \[ 0 \longrightarrow \bigoplus_{i \in I'_n} P^i_n \longrightarrow \bigoplus_{i \in I'_{n-1}} P^i_{n-1} \longrightarrow \cdots \longrightarrow \bigoplus_{i \in I'_1} P^i_1 \longrightarrow \bigoplus_{i \in I'_0} P^i_0 \longrightarrow N' \longrightarrow 0 \] where $N' := {\rm CoKer}\left( \bigoplus_{i \in I'_1} P^i_1 \longrightarrow \bigoplus_{i \in I'_0} P^i_0 \right)$ and $N \subseteq N' \subseteq M$.  \\

Now suppose that $N$ has a small and nice subresolution \[ 0 \longrightarrow \bigoplus_{i \in I'^N_n} P^i_n \longrightarrow \cdots \longrightarrow \bigoplus_{i \in I^N_1} P^i_1 \longrightarrow \bigoplus_{i \in I^N_0} P^i_0 \longrightarrow N \longrightarrow 0 \] of ($\ast$). Take the quotient of ($\ast$) by this resolution and get \[ 0 \longrightarrow \bigoplus_{i \in I_n - I^N_n} P^i_n \longrightarrow \cdots \longrightarrow \bigoplus_{i \in I_1 - I^N_1} P^i_1 \longrightarrow \bigoplus_{i \in I_0 - I^N_0} P^i_0 \longrightarrow M / N \longrightarrow 0. \] Repeat the argument above using this sequence and the small submodule $\left< z + N \right>$, where $z \not\in N$. Then we get a projective subresolution of the previous one: \[ 0 \longrightarrow \bigoplus_{i \in I'_n - I^N_n} P^i_n \longrightarrow \cdots \longrightarrow \bigoplus_{i \in I'_1 - I^N_1} P^i_1 \longrightarrow \bigoplus_{i \in I'_0 - I^N_0} P^i_0 \longrightarrow N' / N \longrightarrow 0 \] where each set $I'_k - I^N_k$ is a small set. As we did in the proof of Lemma \ref{Kaplansky}, we have that \[ 0 \longrightarrow \bigoplus_{i \in I'_n} P^i_n \longrightarrow \cdots \longrightarrow \bigoplus_{i \in I'_1} P^i_1 \longrightarrow \bigoplus_{i \in I'_0} P^i_0 \longrightarrow N' \longrightarrow 0 \] is small and nice subresolution of ($\ast$), and that contains the resolution of $N$ as a nice subresolution. \\
\end{proof}

\begin{lemma} Let $X \in {\rm dw}\widetilde{\mathcal{P}_n}$ and let $Y$ be a bounded above subcomplex of $X$ such that ${\rm Card}(Y_k) \leq \kappa$ for every $k \in \mathbb{Z}$. Then there exists a (bounded above) subcomplex $Y'$ of $X$ such that $Y \subseteq Y'$ and $Y' \in {\rm dw}\widetilde{\mathcal{P}^{\leq \kappa}_n}$. 
\end{lemma}
\begin{proof} We are given the following commutative diagram
\[ \begin{tikzpicture}
\matrix (m) [matrix of math nodes, row sep=1.5em, column sep=3em]
{ Y =\mbox{ } \cdots & 0 & Y_m & Y_{m-1} & \cdots \\ X =\mbox{ } \cdots & X_{m+1} & X_m & X_{m-1} & \cdots \\ };
\path[->]
(m-1-1) edge (m-1-2)
(m-1-2) edge (m-1-3) edge (m-2-2)
(m-1-3) edge node[above] {$\partial_m$} (m-1-4) edge (m-2-3)
(m-1-4) edge (m-1-5) edge (m-2-4)
(m-2-1) edge (m-2-2) (m-2-2) edge node[above] {$\partial_{m+1}$} (m-2-3) (m-2-3) edge node[above] {$\partial_m$} (m-2-4) (m-2-4) edge (m-2-5);
\end{tikzpicture} \]
Since $X_m$ is an $n$-projective module, we have a projective resolution \[ 0 \longrightarrow \bigoplus_{i \in I_n(m)} P^i_n(m) \longrightarrow \cdots \longrightarrow \bigoplus_{i \in I_1(m)} P^i_1(m) \longrightarrow \bigoplus_{i \in I_0(m)} P^i_0(m) \longrightarrow X_m \longrightarrow 0. \] By the previous lemma, there exists a submodule $Y'_m$ of $X_m$ containing $Y_m$, along with a small and nice subresolution \[ 0 \longrightarrow \bigoplus_{i \in I'_n(m)} P^i_n(m) \longrightarrow \cdots \longrightarrow \bigoplus_{i \in I'_1(m)} P^i_1(m) \longrightarrow \bigoplus_{i \in I'_0(m)} P^i_0(m) \longrightarrow Y'_m \longrightarrow 0. \] Note that ${\rm Card}(\partial_m(Y'_m) + Y_{m-1}) \leq \kappa$ and $Y_{m-1} \subseteq \partial_m(Y'_m) + Y_{m-1} \subseteq X_{m-1}$. Now choose a submodule $Y'_{m-1} \subseteq X_{m-1}$ such that $\partial_m(Y'_m) + Y_{m-1} \subseteq Y'_{m-1}$ and $Y'_{m-1}$ has a small and nice subresolution of a fixed resolution of $X_{m-1}$. Repeat this process infinitely many times in order to obtain a complex \[ Y' = \cdots \longrightarrow 0 \longrightarrow Y'_m \longrightarrow Y'_{m-1} \longrightarrow \cdots \] such that $Y \subseteq Y' \subseteq X$ and $Y' \in {\rm dw}\widetilde{\mathcal{P}^{\leq \kappa}_n}$. \\
\end{proof}

\begin{theorem} Every $X \in {\rm ex}\widetilde{\mathcal{P}_n}$ has a ${\rm ex}\widetilde{\mathcal{P}_n^{\leq \kappa}}$-filtration. 
\end{theorem}
\begin{proof} Let $X \in {\rm ex}\widetilde{\mathcal{P}_n}$. We construct a ${\rm ex}\widetilde{\mathcal{P}_n^{\leq \kappa}}$-filtration of $X$ using transfinite induction. For $\alpha = 0$ set $X^0 = 0$. For the case $\alpha = 1$ let $m \in \mathbb{Z}$ be arbitrary and let $T_1 \subseteq X_m$ be a small submodule of $X_m$. Then there exists a small submodule $Y^1_m$ of $X_m$ such that $T_1 \subseteq Y^1_m$ and that $Y^1_m$ has a small and nice projective subresolution of a given resolution of $X_m$. Note that $\partial_m(Y^1_m)$ is a submodule of $X_{m-1}$ with cardinality $\leq \kappa$, so there exists a submodule $Y^1_{m-1}$ of $X_{m-1}$ such that $\partial_m(Y^1_m) \subseteq Y^1_{m-1}$ and that $Y^1_{m-1}$ has a small and nice projective subresolution of a given resolution of $X_{m-1}$. Keep repeating this argument infinitely many times. We obtain a complex \[ Y^1 = \cdots \longrightarrow 0 \longrightarrow Y^1_m \longrightarrow Y^1_{m-1} \longrightarrow \cdots \] which is a subcomplex of $X$ and $Y^1 \in {\rm dw}\widetilde{\mathcal{P}^{\leq \kappa}_n}$.  Note that $Y^1$ is not necessarily exact. We shall construct a complex $X^1$ from $Y^1$ such that $X^1 \subseteq X$ and $X^1 \in {\rm ex}\widetilde{\mathcal{P}^{\leq \kappa}_n}$. The rest of this proof uses an argument similar to the one used in \cite[Theorem 4.6]{Rada}. Fix any $p \in \mathbb{Z}$. Then ${\rm Card}(Y^1_p) \leq \kappa$ and so ${\rm Card}(Z_p Y^1) \leq \kappa$. Since $X$ is exact and ${\rm Card}(Z_p Y^1) \leq \kappa$, there exists a submodule $U \subseteq X_{p+1}$ with ${\rm Card}(U) \leq \kappa$ such that $Z_p Y^1 \subseteq \partial_{p+1}(U)$. Let $C^1$ be a small subcomplex of $X$ such that $U \subseteq C_{p+1}$, $C_j = 0$ for every $j > p+1$, and that each $C_j$ with $j \leq p$ has a small and nice projective subresolution of a given resolution of $X_j$. Since $Y^1 + C$ is a bounded above subcomplex of $X$, there exists a small subcomplex $Y^2$ of $X$ such that $Y^1 + C \subseteq Y^2$ and that each $Y^2_j$ has a small and nice projective subresolution of a given resolution of $X_j$. Note that $Z_p Y^1 \subseteq \partial_{p+1}(Y^2_{p+1})$. Construct $Y^3$ from $Y^2$ as we just constructed $Y^2$ from $Y^1$, and so on, making sure to use the same $p \in \mathbb{Z}$ at each step. Set $X^1 = \bigcup_{j = 1}^\infty Y^j \subseteq X$. Note that $X^1$ is exact at $p$. Repeat this argument to get exactness at any level. So we may assume that $X^1$ is an exact complex.  Every $X^1_k$ has a small and nice projective subresolution of the given resolution of $X_k$. For every $j$ one has a projective subresolution of the form \[ 0 \longrightarrow \bigoplus_{i \in I^j_{n}(k)} P^i_n(k) \longrightarrow \cdots \longrightarrow \bigoplus_{i \in I^j_1(k)} P^i_1(k) \longrightarrow \bigoplus_{i \in I^j_0(k)} P^i_0(k) \longrightarrow Y^j_k \longrightarrow 0, \] where $I^1_l(k) \subseteq I^2_l(k) \subseteq \cdots$ for every $0\leq l \leq n$, by Lemma \ref{lemma2}. If we take the union of all of the previous sequences, then we obtain the following exact sequence: \[ 0 \longrightarrow \bigoplus_{i \in \bigcup_{j\geq i} I^j_{n}(k)} P^i_n(k) \longrightarrow \cdots \longrightarrow \bigoplus_{i \in \bigcup_{j \geq 1} I^j_0(k)} P^i_0(k) \longrightarrow \bigcup_{j \geq 1} Y^j_k = X^1_j \longrightarrow 0, \] where $\bigcup_{j \geq 1} I^j_l(k) \subseteq I_l(k)$ for every $0 \leq l \leq n$. Therefore, $X^1\in {\rm ex}\widetilde{\mathcal{P}^{\leq \kappa}_n}$. Now consider the quotient complex \[ X / X^1 = \cdots \longrightarrow X_{k+1} / X^1_{k+1} \longrightarrow X_k / X^1_k \longrightarrow X_{k-1} / X^1_{k-1} \longrightarrow \cdots. \] Note that each $X_k / X^1_k$ is $n$-projective and that $X / X^1$ is exact. We apply the same procedure above to the complex $X / X^1$ in order to get a complex $X^2 / X^1 \subseteq X / X^1$ such that $X^2 / X^1 \in {\rm ex}\widetilde{\mathcal{P}_n^{\leq \kappa}}$. Note that $X^2$ is an exact complex since the class of exact complexes is closed under extensions, and so $X^2 \in {\rm ex}\widetilde{\mathcal{P}_n}$. The rest of the proof follows by using transfinite induction. \\
\end{proof}

It follows by Proposition \ref{propo} and the Eklof and Trlifaj Theorem that $({\rm ex}\widetilde{\mathcal{P}_n}, ({\rm ex}\widetilde{\mathcal{P}_n})^\perp)$ is a complete cotorsion pair cogenerated by ${\rm ex}\widetilde{\mathcal{P}_n^{\leq \kappa}}$.


\section{The staircase zig-zag argument}

In this section, we present another method to prove that the induced cotorsion pairs $({\rm dw}\widetilde{\mathcal{P}_n}, ({\rm dw}\widetilde{\mathcal{P}_n})^\perp)$ and $({\rm ex}\widetilde{\mathcal{P}_n}, ({\rm ex}\widetilde{\mathcal{P}_n})^\perp)$ are complete. This method also applies for the cotorsion pairs $({\rm dw}\widetilde{\mathcal{F}_n}, ({\rm dw}\widetilde{\mathcal{F}_n})^\perp)$ and $({\rm ex}\widetilde{\mathcal{F}_n}, ({\rm ex}\widetilde{\mathcal{F}_n})^\perp)$. From now on, let $\mathcal{A}$ denote either the class of projective modules or the class of flat modules, and let $\mathcal{A}_n$ denote the class of all modules $M$ having an $\mathcal{A}$-resolution of length $n$: \[ 0 \longrightarrow A_n \longrightarrow \cdots \longrightarrow A_1 \longrightarrow A_0 \longrightarrow M \longrightarrow 0, \] where $A_k \in \mathcal{A}$ for every $0 \leq k \leq n$. Note that $\mathcal{A}_0 = \mathcal{A}$. \\

As we mentioned before, in \cite{Gobel} it is proven that $(\mathcal{F}_n, \mathcal{F}^\perp_n)$ is a complete cotorsion pair. We give a simpler proof of this fact. \\

\begin{lemma}\label{lemma3} Let $M \in \mathcal{F}_n$ with a flat resolution \[ 0 \longrightarrow F_n \stackrel{f_n}\longrightarrow F_{n-1} \longrightarrow \cdots \longrightarrow F_1 \stackrel{f_1}\longrightarrow F_0 \stackrel{f_0}\longrightarrow M \longrightarrow 0 \mbox{ \ (1)} \] and let $N$ be a small submodule of $M$. Then there exists a flat subresolution \[ 0 \longrightarrow S'_n \longrightarrow \cdots \longrightarrow S'_1 \longrightarrow S'_0 \longrightarrow N' \longrightarrow 0 \] of (1) such that $S'_k$ is a small and pure submodule of $F_k$, for every $0 \leq k \leq n$, and such that $N \subseteq N'$. Moreover, if $N$ has a subresolution of (1) \[ 0 \longrightarrow S_n \longrightarrow \cdots \longrightarrow S_1 \longrightarrow S_0 \longrightarrow N \longrightarrow 0 \] where $S_k$ is a small and pure submodule of $F_k$, for every $0 \leq k \leq n$, then the above resolution of $N'$ can be constructed in such a way that it contains the resolution of $N$. 
\end{lemma}
\begin{proof} First, note that for every flat module $F$ and for every small submodule $N \subseteq F$, there exists a small and pure submodule $S \subseteq F$ such that $N \subseteq S$ (for a proof of this, see \cite[Lemma 5.3.12]{Enochs}). For every $x \in N$ there exists $y_x \in F_0$ such that $x = f_0(y_x)$. Let $Y$ be the set $\{ y_x \mbox{ : } x \in N \mbox{ and } f_0(y_x) = x \}$ and consider the submodule $\left< Y \right> \subseteq F_0$. Since $\left< Y \right>$ is small, there exists a small pure submodule $S_0(1) \subseteq F_0$ such that $\left< Y \right> \subseteq S_0(1)$. Note that $f_0(S_0(1)) \supseteq N$. Now consider ${\rm Ker}\left( f_0|_{S_0(1)} \right)$ and let $A$ be a set of preimages of ${\rm Ker}\left( f_0|_{S_0(1)} \right)$ such that $f_1\left( \left< A \right> \right) \supseteq {\rm Ker}\left( f_0|_{S_0(1)} \right)$. It is easy to see that $\left< A \right>$ is a small submodule of $F_1$, so we can embed it into a small pure submodule $S_1(1) \subseteq F_1$. Hence we have $f_1(S_1(1)) \supseteq {\rm Ker}\left( f_0|_{S_0(1)} \right)$. Now consider ${\rm Ker}\left( f_1|_{S_1(1)} \right)$ and repeat the same process above in order to find a small pure submodule $S_2(1) \subseteq F_2$ such that $f_2(S_2(1)) \supseteq {\rm Ker}\left( f_1|_{S_1(1)} \right)$. Keep doing this until find a small pure submodule $S_n(1) \subseteq F_n$ such that $f_n(S_n(1)) \supseteq {\rm Ker}\left( f_{n-1}|_{S_{n-1}(1)} \right)$. Now $f_n(S_n(1))$ is a small submodule of $F_{n-1}$, so there is a small pure submodule $S_{n-1}(2) \subseteq F_{n-1}$ such that $f_n(S_n(1)) \subseteq S_{n-1}(2)$. Repeat this process until find a small pure submodule $S_0(2) \subseteq F_0$ such that $f_1(S_1(2)) \subseteq S_0(2)$. If we now consider ${\rm Ker}\left( f_0|_{S_0(2)} \right) \subseteq F_0$, we repeat the same argument above to find a small pure submodule $S_1(3) \subseteq F_1$ such that $f_1(S_1(3)) \supseteq {\rm Ker}\left( f_0|_{S_0(2)} \right)$. Keep repeating this zig-zag procedure infinitely many times and set $S_k = \bigcup_{i \geq 1} S_k(i)$, for every $0 \leq k \leq n$. Note that each $S_k$ is a pure submodule of $F_k$. By construction, we get an exact complex \[ 0 \longrightarrow S_n \longrightarrow S_{n-1} \longrightarrow \cdots \longrightarrow S_1 \longrightarrow S_0 \longrightarrow Q \longrightarrow 0, \mbox{ \ \ \ (2)} \] where $Q = {\rm CoKer}(f_1|_{S_1}) \subseteq M$. If we take the quotient of (1) by (2), we get an exact complex \[ 0 \longrightarrow F_n / S_n \longrightarrow F_{n-1} / S_{n-1} \longrightarrow \cdots \longrightarrow F_1 / S_1 \longrightarrow F_0 / S_0 \longrightarrow M / Q \longrightarrow 0. \] \newpage Since each $S_k$ is a pure submodule of $F_k$, we know that $S_k$ and $F_k / S_k$ are flat modules. Therefore, $Q$ is a small $n$-flat submodule with $N \subseteq Q$ such that $M / Q$ is also $n$-flat. The rest of the proof follows as in Lemma \ref{lemma2}. \\ 
\end{proof}

\begin{remark} From Lemma \ref{lemma2} and Lemma \ref{lemma3}, we have that for every $A \in \mathcal{A}_n$ and for every small submodule $0 \neq N \subseteq A$, there exists a small submodule $A' \subseteq A$ in $\mathcal{A}$ such that $N \subseteq A'$ and $A / A' \in \mathcal{A}$. \\
\end{remark}

\begin{theorem}\label{metodo2} Let $X \in {\rm ex}\widetilde{\mathcal{A}_n}$ and let $x \in X$ (i.e. $x \in X_m$ for some $m \in \mathbb{Z}$). Then there exists a small complex $Y \in {\rm ex}\widetilde{\mathcal{A}^{\leq \kappa}_n}$ such that $x \in Y$ and $X / Y \in {\rm ex}\widetilde{\mathcal{A}_n}$.  \\
\end{theorem}

The following proof is based on an argument given in \cite[Proposition 4.1]{Aldrich1}, where the authors prove that $({\rm dw}\widetilde{\mathcal{F}_0}, ({\rm dw}\widetilde{\mathcal{F}_0})^\perp)$ and $({\rm ex}\widetilde{\mathcal{F}_0}, ({\rm ex}\widetilde{\mathcal{F}_0})^\perp)$ are complete cotorsion pairs. We shall call this argument the {\bf staircase zig-zag}. \\

\begin{proof}[Proof of Theorem \ref{metodo2}] Assume without loss of generality that $x \in X_0$. Consider the submodule $\left< x \right> \subseteq X_0$. Since $X_0 \in \mathcal{A}$ and $\left< x \right>$ is small, we can embed $\left< x \right>$ into a submodule $\mathcal{A}^{\leq \kappa} \ni Y^1_0 \subseteq X_0$ such that $X_0 / Y^1_0 \in \mathcal{A}$. Since $X$ is exact, we can construct a small and exact subcomplex \[ L^1 \mbox{ = } \cdots L^1_2 \longrightarrow L^1_1 \longrightarrow Y^1_0 \longrightarrow \partial_0(Y^1_0) \longrightarrow 0 \longrightarrow \cdots. \] Since $\partial_0(Y^1_0)$ is small, there exists a submodule $\mathcal{A}^{\leq \kappa} \ni Y_{-1}^2 \subseteq X_{-1}$ such that $X_{-1} / Y_{-1}^2 \in \mathcal{A}$. Now construct a small and exact subcomplex \[ L^2 \mbox{ = } \cdots L^2_2 \longrightarrow L^2_1 \longrightarrow L^2_0 \longrightarrow Y^2_{-1} \longrightarrow \partial_{-1}(Y^2_{-1}) \longrightarrow 0 \longrightarrow \cdots. \] \newpage

Note that it is possible to construct $L^2$ containing $L^1$. Now embed $L^2_0$ into a submodule $\mathcal{A}^{\leq \kappa} \ni Y^3_0 \subseteq X_0$ such that $X_0 / Y^3_0 \in \mathcal{A}$. Again, construct a small and exact subcomplex \[ L^3 \mbox{ = } \cdots \longrightarrow L^3_2 \longrightarrow L^3_1 \longrightarrow Y^3_0 \longrightarrow Y^2_{-1} + \partial_0(Y_0^3) \longrightarrow \partial_{-1}(Y_{-1}^2) \longrightarrow 0 \longrightarrow \cdots \] containing $L^2$. Now let $Y^4_1 \in \mathcal{A}^{\leq \kappa}$ be a submodule of $X_1$ containing $L^3_1$ such that $X_1 / Y^4_1 \in \mathcal{A}$, and construct an exact and small complex \[ L^4 \mbox{ = } \cdots \longrightarrow L^4_2 \longrightarrow Y^4_1 \longrightarrow Y^3_0 + \partial_1(Y_1^4) \longrightarrow Y^2_{-1} + \partial_0(Y_0^3) \longrightarrow \partial_{-1}(Y^2_{-1}) \longrightarrow 0 \longrightarrow \cdots \] containing $L^3$. Embed $Y^3_0 + \partial_1(Y^4_0)$ into a submodule $\mathcal{A}^{\leq \kappa} \ni Y^5_0 \subseteq X_0$ such that $X_0 / Y^5_0 \in \mathcal{A}$. Construct an exact and small subcomplex \[ L^5 \mbox{ = } \cdots \longrightarrow L^5_2 \longrightarrow L^5_1 \longrightarrow Y^5_0 \longrightarrow Y^2_{-1} + \partial_0(Y^5_0) \longrightarrow \partial_{-1}(Y^2_{-1}) \longrightarrow 0 \longrightarrow \cdots \] containing $L^4$. In a similar way, construct small and exact complexes 
\begin{align*}
L^6 & = \cdots \longrightarrow L^6_1 \longrightarrow L^6_0 \longrightarrow Y^6_{-1} \longrightarrow \partial_{-1}(Y^6_{-1}) \longrightarrow 0 \longrightarrow \cdots, \\
L^7 & = \cdots \longrightarrow L^7_1 \longrightarrow L^7_0 \longrightarrow L^7_{-1} \longrightarrow Y^7_{-2} \longrightarrow \partial_{-2}(Y_{-2}^7) \longrightarrow 0 \longrightarrow \cdots,
\end{align*}
such that $Y_{-1}^6 \in \mathcal{A}$ is a small submodule of $X_{-1}$ containing $Y^6_{-1} + \partial_0(Y^5_0)$, and $Y_{-2}^7 \in \mathcal{A}$ is a small submodule of $X_{-2}$ containing $\partial_{-1}(Y_{-1}^6)$. We have the following commutative diagram of subcomplexes of $X$: \newpage
\[ \begin{tikzpicture}
\matrix (m) [matrix of math nodes, row sep=0.6em, column sep=1em]
{ 0 & 0 & 0 & 0 & 0 & 0 \\ L^1_2 & L^1_1 & Y^1_0 & \partial_0(Y^1_0) & 0 & 0 \\ L^2_2 & L^2_1 & L^2_0 & Y^2_{-1} & \partial_{-1}(Y^2_{-1}) & 0 \\ L^3_2 & L^3_1 & Y^3_0 & \partial_0(Y^3_0) + Y^2_{-1} & \partial_{-1}(Y^2_{-1}) & 0 \\ L^4_2 & Y^4_1 & Y^3_0 + \partial_1(Y^4_1) & \partial_0(Y^3_0) + Y^2_{-1} & \partial_{-1}(Y^2_{-1}) & 0 \\ L^5_2 & L^5_1 & Y^5_0 & Y^2_{-1} + \partial_0(Y^5_0) & \partial_{-1}(Y^2_{-1}) & 0 \\ L^6_2 & L^6_1 & L^6_0 & Y^6_{-1} & \partial_{-1}(Y^6_{-1}) & 0 \\ L^7_2 & L^7_1 & L^7_0 & L^7_{-1} & Y^7_{-2} & \partial_{-2}(Y^7_{-2}) \\ L^8_2 & L^8_{1} & L^8_0 & Y^8_{-1} & \partial_{-1}(Y^8_{-1}) + Y^7_{-2} & \partial_{-2}(Y^7_{-2}) \\ L^9_2 & L^9_1 & Y^9_0 & \partial_0(Y^9_0) + Y^8_{-1} & \partial_{-1}(Y^8_{-1}) + Y^7_{-2} & \partial_{-2}(Y^7_{-2}) \\ L^{10}_2 & Y^{10}_{1} & \partial_1(Y^{10}_1) + Y^9_0 & \partial_0(Y^9_0) + Y^8_{-1} & \partial_{-1}(Y^8_{-1}) + Y^7_{-2} & \partial_{-2}(Y^7_{-2}) \\ Y^{11}_2 & \partial_2(Y^{11}_2) + Y^{10}_{1} & \partial_1(Y^{10}_1) + Y^9_0 & \partial_0(Y^9_0) + Y^8_{-1} & \partial_{-1}(Y^8_{-1}) + Y^7_{-2} & \partial_{-2}(Y^7_{-2}) \\ \vdots & \vdots & \vdots & \vdots & \vdots & \vdots \\ };
\path[->]
(m-1-1) edge (m-2-1) (m-1-2) edge (m-2-2) (m-1-3) edge (m-2-3) (m-1-4) edge (m-2-4) (m-1-5) edge (m-2-5) (m-1-6) edge (m-2-6)
(m-2-1) edge (m-2-2) edge (m-3-1) (m-2-2) edge (m-2-3) edge (m-3-2) (m-2-3) edge (m-2-4) edge (m-3-3) (m-2-4) edge (m-2-5) edge (m-3-4) (m-2-5) edge (m-2-6) edge (m-3-5) (m-2-6) edge (m-3-6)
(m-3-1) edge (m-3-2) edge (m-4-1) (m-3-2) edge (m-3-3) edge (m-4-2) (m-3-3) edge (m-3-4) edge (m-4-3) (m-3-4) edge (m-3-5) edge (m-4-4) (m-3-5) edge (m-3-6) edge (m-4-5) (m-3-6) edge (m-4-6)
(m-4-1) edge (m-4-2) edge (m-5-1) (m-4-2) edge (m-4-3) edge (m-5-2) (m-4-3) edge (m-4-4) edge (m-5-3) (m-4-4) edge (m-4-5) edge (m-5-4) (m-4-5) edge (m-4-6) edge (m-5-5) (m-4-6) edge (m-5-6)
(m-5-1) edge (m-5-2) edge (m-6-1) (m-5-2) edge (m-5-3) edge (m-6-2) (m-5-3) edge (m-5-4) edge (m-6-3) (m-5-4) edge (m-5-5) edge (m-6-4) (m-5-5) edge (m-5-6) edge (m-6-5) (m-5-6) edge (m-6-6)
(m-6-1) edge (m-6-2) edge (m-7-1) (m-6-2) edge (m-6-3) edge (m-7-2) (m-6-3) edge (m-6-4) edge (m-7-3) (m-6-4) edge (m-6-5) edge (m-7-4) (m-6-5) edge (m-6-6) edge (m-7-5) (m-6-6) edge (m-7-6)
(m-7-1) edge (m-7-2) edge (m-8-1) (m-7-2) edge (m-7-3) edge (m-8-2) (m-7-3) edge (m-7-4) edge (m-8-3) (m-7-4) edge (m-7-5) edge (m-8-4) (m-7-5) edge (m-7-6) edge (m-8-5) (m-7-6) edge (m-8-6)
(m-8-1) edge (m-8-2) edge (m-9-1) (m-8-2) edge (m-8-3) edge (m-9-2) (m-8-3) edge (m-8-4) edge (m-9-3) (m-8-4) edge (m-8-5) edge (m-9-4) (m-8-5) edge (m-8-6) edge (m-9-5) (m-8-6) edge (m-9-6)
(m-9-1) edge (m-9-2) edge (m-10-1) (m-9-2) edge (m-9-3) edge (m-10-2) (m-9-3) edge (m-9-4) edge (m-10-3) (m-9-4) edge (m-9-5) edge (m-10-4) (m-9-5) edge (m-9-6) edge (m-10-5) (m-9-6) edge (m-10-6)
(m-10-1) edge (m-10-2) edge (m-11-1) (m-10-2) edge (m-10-3) edge (m-11-2) (m-10-3) edge (m-10-4) edge (m-11-3) (m-10-4) edge (m-10-5) edge (m-11-4) (m-10-5) edge (m-10-6) edge (m-11-5) (m-10-6) edge (m-11-6)
(m-11-1) edge (m-11-2) edge (m-12-1) (m-11-2) edge (m-11-3) edge (m-12-2) (m-11-3) edge (m-11-4) edge (m-12-3) (m-11-4) edge (m-11-5) edge (m-12-4) (m-11-5) edge (m-11-6) edge (m-12-5) (m-11-6) edge (m-12-6)
(m-12-1) edge (m-12-2) edge (m-13-1) (m-12-2) edge (m-12-3) edge (m-13-2) (m-12-3) edge (m-12-4) edge (m-13-3) (m-12-4) edge (m-12-5) edge (m-13-4) (m-12-5) edge (m-12-6) edge (m-13-5) (m-12-6) edge (m-13-6);
\end{tikzpicture} \]
where the $k+1$-th complex can be constructed in such a way that it contains the $k$-th complex. Note that the submodules $S^k_i$ appear according to the following pattern: \newpage 
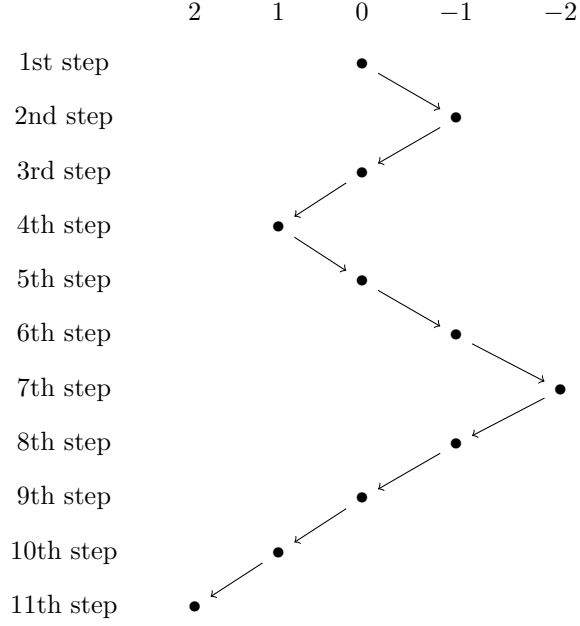
\[ \begin{tikzpicture}
\matrix (m) [matrix of math nodes, row sep=0.5em, column sep=2em]
{ & 2 & 1 & 0 & -1 & -2 \\ 1\mbox{st step} & & & \bullet \\ 2\mbox{nd step} & & & & \bullet \\ 3\mbox{rd step} & & & \bullet \\ 4\mbox{th step} & & \bullet \\ 5\mbox{th step} & & & \bullet \\ 6\mbox{th step} & & & & \bullet \\ 7\mbox{th step} & & & & & \bullet \\ 8\mbox{th step} & & & & \bullet \\ 9\mbox{th step} & & & \bullet \\ 10\mbox{th step} & & \bullet \\ 11\mbox{th step} & \bullet \\ };
\path[->]
(m-2-4) edge (m-3-5) (m-3-5) edge (m-4-4) (m-4-4) edge (m-5-3) (m-5-3) edge (m-6-4) (m-6-4) edge (m-7-5) (m-7-5) edge (m-8-6) (m-8-6) edge (m-9-5) (m-9-5) edge (m-10-4) (m-10-4) edge (m-11-3) (m-11-3) edge (m-12-2);
\end{tikzpicture} \]
Let $Y = \bigcup_{n \geq 1} L^n$, where $Y_i = \bigcup_{n \geq 1} (L^n)_i$. It is clear that $Y$ is an exact complex. We check that $Y$ is also a ${\rm dw}\widetilde{\mathcal{A}_n}$. For example, consider \[ Y_0 = Y^1_0 \cup L^2_0 \cup Y^3_0 \cup (Y^3_0 + \partial_1(Y^4_1)) \cup Y^5_0 \cup \cdots = Y^1_0 \cup Y^3_0 \cup Y^5_0 \cup \cdots. \] It is clear that $Y_0$ is small. At this point, we split the proof in two cases: 

\begin{itemize}
\item[ (1)] $\mathcal{A} = \mathcal{P}_0$: Consider a projective resolution of $X_0$ of length $n$, say \[ 0 \longrightarrow \bigoplus_{i \in I_n} P^i_n \longrightarrow \cdots \longrightarrow \bigoplus_{i \in I_1} P^i_1 \longrightarrow \bigoplus_{i \in I_0} P^i_0 \longrightarrow X_0 \longrightarrow 0 \mbox{ \ (1)}, \] where each direct sum is a direct sum of countably generated projective modules. By Lemma \ref{lemma2}, we can construct $Y^1_0$ containing $\left< x \right>$ with a subresolution \[ 0 \longrightarrow \bigoplus_{i \in I^1_n} P^i_n \longrightarrow \cdots \longrightarrow \bigoplus_{i \in I^1_1} P^i_1 \longrightarrow \bigoplus_{i \in I^1_0} P^i_0 \longrightarrow Y^1_0 \longrightarrow 0 \mbox{ \ (2)}, \] where each $I^1_k$ is a small subset of $I_k$. Note that the quotient of (1) by (2) yields a projective resolution of $X_0 / Y^1_0$ of length $n$, so $X_0 / Y^1_0 \in \mathcal{P}_n$. Using Lemma \ref{lemma2} again, we can construct a subresolution \[ 0 \longrightarrow \bigoplus_{i \in I^3_n} P^i_n \longrightarrow \cdots \longrightarrow \bigoplus_{i \in I^3_1} P^i_1 \longrightarrow \bigoplus_{i \in I^3_0} P^i_0 \longrightarrow Y^3_0 \longrightarrow 0 \mbox{ (3)} \] containing (2) such that $X_0 / Y^3_0 \in \mathcal{P}_n$. We keep applying Lemma \ref{lemma2} to get an ascending chain of subresolutions of (1):
\begin{align*}
0 & \longrightarrow \bigoplus_{i \in I^1_n} P^i_n \longrightarrow \cdots \longrightarrow \bigoplus_{i \in I^1_1} P^i_1 \longrightarrow \bigoplus_{i \in I^1_0} P^i_0 \longrightarrow Y^1_0 \longrightarrow 0 \\
0 & \longrightarrow \bigoplus_{i \in I^3_n} P^i_n \longrightarrow \cdots \longrightarrow \bigoplus_{i \in I^3_1} P^i_1 \longrightarrow \bigoplus_{i \in I^3_0} P^i_0 \longrightarrow Y^3_0 \longrightarrow 0 \\
0 & \longrightarrow \bigoplus_{i \in I^5_n} P^i_n \longrightarrow \cdots \longrightarrow \bigoplus_{i \in I^5_1} P^i_1 \longrightarrow \bigoplus_{i \in I^5_0} P^i_0 \longrightarrow Y^5_0 \longrightarrow 0 \\
& \ \ \ \ \ \ \ \ \ \ \ \ \ \ \ \ \ \ \ \ \ \ \ \ \ \ \ \ \ \ \ \ \ \ \vdots 
\end{align*}
Now we take the union of this ascending chain and get an exact complex
\begin{align*}
0 & \longrightarrow \bigcup_j \bigoplus_{i \in I^j_n} P^i_n \longrightarrow \cdots \longrightarrow \bigcup_j \bigoplus_{i \in I^j_1} P^i_1 \longrightarrow \bigcup_j \bigoplus_{i \in I^j_0} P^i_0 \longrightarrow \bigcup_j Y^j_0 \longrightarrow 0 \\
& \ \ \ \ \ \ \ \ \ \ \ \ \ \ \ \ \ \ \ \ \ \ \ \ \ \ \ \ \ \ \ \ \ \ \ \ \ \ \ \ \ \ = \\
0 & \longrightarrow \bigoplus_{i \in \bigcup_j I^j_n} P^i_n \longrightarrow \cdots \longrightarrow \bigoplus_{i \in \bigcup_j I^j_1} P^i_1 \longrightarrow \bigoplus_{i \in \bigcup_j I^j_0} P^i_0 \longrightarrow Y_0 \longrightarrow 0 \mbox{ \ (4)}
\end{align*}
\newpage Since each $\bigcup_j I^j_k$ is a small subset of $I_k$, we have that the previous sequence is a $\mathcal{P}^{\leq \kappa}_0$-subresolution of (1). Note also that the quotient of (1) by (4) yields a projective resolution of $X_0 / Y_0$ of length $n$. Then $Y_0 \in \mathcal{P}^{\leq \kappa}_n$. In a similar way, we can show that $Y_m \in \mathcal{P}^{\leq \kappa}_n$ and $X_m / Y_m \in \mathcal{P}_n$, for every $m \in \mathbb{Z}$. Hence, $Y \in {\rm ex}\widetilde{\mathcal{P}^{\leq \kappa}_n}$. Note also that the quotient of exact complexes is exact, so we also have $X / Y \in {\rm ex}\widetilde{\mathcal{P}^{\leq \kappa}_n}$. 

\item[ (2)] $\mathcal{A} = \mathcal{F}_0$: Consider a flat resolution of $X_0$ of length $n$, say \[ 0 \longrightarrow F_n \longrightarrow \cdots \longrightarrow F_1 \longrightarrow F_0 \longrightarrow X_0 \longrightarrow 0 \mbox{ \ (1')}. \] By Lemma \ref{lemma3}, we can construct a subresolution \[ 0 \longrightarrow S^1_n \longrightarrow \cdots \longrightarrow S^1_1 \longrightarrow S^1_0 \longrightarrow Y^1_0 \longrightarrow 0, \] where $\left< x \right> \subseteq Y^1_0$, and each $S^1_k$ is a small and pure submodule of $F_k$. As we did in the previous case, applying Lemma \ref{lemma3} infinitely many times, we can get an ascending chain of subresolutions
\begin{align*}
0 & \longrightarrow S^1_n \longrightarrow \cdots \longrightarrow S^1_1 \longrightarrow S^1_0 \longrightarrow Y^1_0 \longrightarrow 0 \\
0 & \longrightarrow S^3_n \longrightarrow \cdots \longrightarrow S^3_1 \longrightarrow S^3_0 \longrightarrow Y^3_0 \longrightarrow 0 \\
0 & \longrightarrow S^5_n \longrightarrow \cdots \longrightarrow S^5_1 \longrightarrow S^5_0 \longrightarrow Y^5_0 \longrightarrow 0 \\
& \ \ \ \ \ \ \ \ \ \ \ \ \ \ \ \ \ \ \ \ \ \ \ \ \ \ \ \ \vdots 
\end{align*}
Taking the union of these subresolutions yields an exact sequence \[ 0 \longrightarrow \bigcup_j S^j_n \longrightarrow \cdots \longrightarrow \bigcup_j S^j_1 \longrightarrow \bigcup_j S^j_0 \longrightarrow Y_0 \longrightarrow 0 \mbox{ \ (2')}, \] where each $\bigcup_j S^j_k$ is a small and pure submodule of $F_k$, and so it is flat and the quotient $F_k / \bigcup_j S^j_k$ is also flat. Then we have that $Y_0 \in \mathcal{F}^{\leq \kappa}_n$ and $X_0 / Y_0 \in \mathcal{F}_n$ (take the quotient of (1') by (2')). In a similar way, we can show that $Y_m \in \mathcal{F}^{\leq \kappa}_n$ and $X_m / Y_m \in \mathcal{F}_n$, for every $m \in \mathcal{Z}$. It follows $Y \in {\rm ex}\widetilde{\mathcal{F}^{\leq \kappa}_n}$ and $X / Y \in {\rm ex}\widetilde{\mathcal{F}_n}$. 
\end{itemize}
In both cases, we have constructed a subcomplex $Y \subseteq X$ containing $x$ such that $Y \in {\rm ex}\widetilde{\mathcal{A}^{\leq \kappa}_n}$ and $X / Y \in {\rm ex}\widetilde{\mathcal{A}_n}$. \\
\end{proof}

\begin{remark} Note that using the staircase zig-zag method, there is no need to assume that $R$ is right noetherian. 
\end{remark}


\section{dw-$n$-projective and dw-$n$-flat model structures}

In this section we obtain two model structures on $\Ch$, one from the two complete cotorsion pairs $({\rm dw}\widetilde{\mathcal{P}_n}, ({\rm dw}\widetilde{\mathcal{P}_n})^\perp)$ and $({\rm ex}\widetilde{\mathcal{P}_n}, ({\rm ex}\widetilde{\mathcal{P}_n})^\perp)$, and the other one from $({\rm dw}\widetilde{\mathcal{F}_n}, ({\rm dw}\widetilde{\mathcal{F}_n})^\perp)$ and $({\rm ex}\widetilde{\mathcal{F}_n}, ({\rm ex}\widetilde{\mathcal{F}_n})^\perp)$. \\

Recall that a subcategory $\mathcal{D}$ of an abelian category $\mathcal{C}$ is said to be {\bf thick} if the following two conditions are satisfied:
\begin{itemize}
\item[ (a)] $\mathcal{D}$ is {\bf closed under retracts}, i.e., given a sequence \[ D' \stackrel{f}\longrightarrow D \stackrel{g}\longrightarrow D' \] with $g \circ f = {\rm id}_{D'}$ and $D \in \mathcal{D}$, then $D' \in \mathcal{B}$. 

\item[ (b)] If two out of three of the terms in a short exact sequence \[ 0 \longrightarrow D'' \longrightarrow D \longrightarrow D' \longrightarrow 0 \] are in $\mathcal{D}$, then so is the third. \\
\end{itemize} 
For example, the class $\mathcal{E}$ of exact complexes is thick. \\

The following theorem by Hovey describes how to get an abelian model structure from two complete cotorsion pairs. 

\begin{theorem}{\rm \cite[Theorem 2.2]{Hovey2}}\label{hovey} Let $\mathcal{C}$ be a bicomplete abelian category with enough projective and injective objects, and let $(\mathcal{A}, \mathcal{B} \cap \mathcal{E})$ and $(\mathcal{A} \cap \mathcal{E}, \mathcal{B})$ be two complete cotorsion pairs in $\mathcal{C}$ such that the class $\mathcal{E}$ is thick. Then there is a unique abelian model structure on $\mathcal{C}$ such that $\mathcal{A}$ is the class of cofibrant objects, $\mathcal{B}$ is the class of fibrant objects, and $\mathcal{E}$ is the class of trivial objects. \\ 
\end{theorem}

Cotorsion pairs of the form $(\mathcal{A}, \mathcal{B} \cap \mathcal{E})$ and $(\mathcal{A} \cap \mathcal{E}, \mathcal{B})$ are called {\bf compatible} by Gillespie in \cite{Gillespie}. We shall use the following result from \cite{Rada} to show that $({\rm dw}\widetilde{\mathcal{P}_n}, ({\rm dw}\widetilde{\mathcal{P}_n})^\perp)$ and $({\rm ex}\widetilde{\mathcal{P}_n}, ({\rm ex}\widetilde{\mathcal{P}_n})^\perp)$ are compatible cotorsion pairs. 

\begin{lemma}{\rm \cite[Lemma 5.3]{Rada}} If $(\mathcal{C}, \mathcal{D}')$ is a cotorsion pair and $(\mathcal{U, V})$ is a complete and hereditary cotorsion pair in $\Ch$ (i.e. $\mathcal{V}$ contains the injective complexes and it is closed under extensions and under taking cokernels of monomorphisms), and if $\mathcal{U} \subseteq \mathcal{C}$ then when $(\mathcal{C} \cap \mathcal{V})^\perp = \mathcal{D}$, we have $\mathcal{D}' = \mathcal{D} \cap \mathcal{V}$. \\
\end{lemma}

In the previous lemma, put $(\mathcal{C}, \mathcal{D}') = ({\rm dw}\widetilde{\mathcal{A}_n}, ({\rm dw}\widetilde{\mathcal{A}_n})^\perp)$ and $(\mathcal{U}, \mathcal{V}) = (\mbox{}^\perp\mathcal{E}, \mathcal{E})$. In \cite[Lemma 5.1]{Rada} it is proven that $(\mbox{}^\perp\mathcal{E}, \mathcal{E})$ is cogenerated by a set, so it is complete. In \cite{Xu}, it is proven that $\mbox{}^\perp\mathcal{E}$ is the class of dg-projective complexes, where a complex $X$ is dg-projective if $X_m \in \mathcal{P}_0$ for every $m \in \mathbb{Z}$ and every map $X \longrightarrow Y$ is nullhomotopic whenever $Y$ is exact. Hence, it is clear that $\mbox{}^\perp\mathcal{E} \subseteq {\rm dw}\widetilde{\mathcal{A}_n}$, i.e. $\mathcal{U} \subseteq \mathcal{C}$. Setting $\mathcal{D} = (\mathcal{C} \cap \mathcal{V})^\perp = ({\rm dw}\widetilde{\mathcal{A}_n} \cap \mathcal{E})^\perp$, we have $({\rm dw}\widetilde{\mathcal{A}_n})^\perp = ({\rm dw}\widetilde{\mathcal{A}_n} \cap \mathcal{E})^\perp \cap \mathcal{E}$. So we obtain \\
\begin{align*}
({\rm dw}\widetilde{\mathcal{A}_n}, ({\rm dw}\widetilde{\mathcal{A}_n})^\perp) & = ({\rm dw}\widetilde{\mathcal{A}_n}, ({\rm dw}\widetilde{\mathcal{A}_n} \cap \mathcal{E})^\perp \cap \mathcal{E}), \\
({\rm ex}\widetilde{\mathcal{A}_n}, ({\rm ex}\widetilde{\mathcal{A}_n})^\perp) & = ({\rm dw}\widetilde{\mathcal{A}_n} \cap \mathcal{E}, ({\rm dw}\widetilde{\mathcal{A}_n} \cap \mathcal{E})^\perp). \\
\end{align*}

From Theorem \ref{hovey} and the previous equalities, we have:
\begin{corollary} There exists a unique abelian model structure in $\Ch$ such that ${\rm dw}\widetilde{\mathcal{A}_n}$ is the class of cofibrant objects, $({\rm ex}\widetilde{\mathcal{A}_n})^\perp$ is the class of fibrant objects, and $\mathcal{E}$ is the class of trivial objects. In the case $\mathcal{A} = \mathcal{P}_0$ we name this structure the {\bf dw-$n$-projective model structure}, and when $\mathcal{A} = \mathcal{F}_0$ we name it the {\bf dw-$n$-flat model structure}. 
\end{corollary}

\begin{remark} As far as the author knows, the dw-projective model structure first appeared in \cite{Rada}, while the dw-flat model did in \cite{Gillespie}.
\end{remark}


\end{document}